\documentclass[reqno, 12pt]{amsart}
\usepackage{amsmath}
 \usepackage{amssymb}
\usepackage{amsthm}
\usepackage{times}
\usepackage{latexsym}
\usepackage[mathscr]{eucal}

\numberwithin{equation}{section}
 
  \newtheorem{theorem}{Theorem}[section]
  \newtheorem{proposition}[theorem]{Proposition}
  \newtheorem{lemma}[theorem]{Lemma}
  \newtheorem{corollary}[theorem]{Corollary}

\title[Warped products with a Tripathi connection]{Warped products with a Tripathi connection}

\author[A. S. Diallo, F. Massamba and S. J. Mbatakou]{Abdoul Salam Diallo*, Fortun\'{e} Massamba** and Salomon Joseph Mbatakou***}

\newcommand{\acr}{\newline\indent}

\address{\llap{*\,} School of Mathematics, Statistics and Computer Science\acr
 University of KwaZulu-Natal\acr
 Private Bag X01, Scottsville 3209\acr
South Africa  \acr
and \acr
Universit\'e Alioune Diop de Bambey\acr
UFR SATIC, D\'epartement de Math\'ematiques\acr
B. P. 30, Bambey, S\'en\'egal}
\email{Diallo@ukzn.ac.za, abdoulsalam.diallo@uadb.edu.sn}
\address{\llap{**\,} School of Mathematics, Statistics and Computer Science\acr
 University of KwaZulu-Natal\acr
 Private Bag X01, Scottsville 3209\acr
South Africa}
\email{massfort@yahoo.fr, Massamba@ukzn.ac.za}
\address{\llap{***\,} School of Mathematics, Statistics and Computer Science\acr
 University of KwaZulu-Natal\acr
 Private Bag X01, Scottsville 3209\acr
South Africa}
\email{Mbatakous@ukzn.ac.za,mbatakou@gmail.com}

\thanks{} 



\begin{document}
 
\begin{abstract}  
The warped product $M_1 \times_F M_2$ of two Riemannian manifolds $(M_1,g_1)$ and $(M_2,g_2)$ is the product manifold $M_1 \times M_2$ equipped with the warped product metric $g=g_1 + F^2 g_2$, where $F$ is a positive function on $M_1$. The notion of warped product manifolds is one of the most fruitful generalizations of Riemannian products. Such a notion plays very important roles in differential geometry as well as in physics, especially in general relativity. In this paper we study warped product manifolds endowed with a Tripathi connection. We establish some relationships between the Tripathi connection of the warped product $M$ to those $M_1$ and $M_2$.
\end{abstract}

\maketitle

\section{Introduction}
In 1969, R. L. Bishop and B. O'Neil \cite{bishop} introduced the concept of warped products, which were used to construct a large class of complete Riemannian manifolds with negative sectional curvature. Warped product have significant applications, in general relativity, in the studies related to solution of Einstein's equations. Besides general relativity, warped product structures have also generated interest in many areas of geometry, especially due to their role in construction of new examples with interesting curvature and symmetry properties.\\
Let $(M_1,g_1)$ and $(M_2,g_2)$ be two Riemannian manifolds and $F$ be a positive differentiable functions on $M_1$. Consider the product manifold $M_1\times M_2$ with its projections $\pi:M_1\times M_2\rightarrow M_1$ and 
$\sigma:M_1\times M_2\rightarrow M_2$. The \textit{warped product} is the product manifold $M_1\times M_2$ with the Riemannian structure $g=g_1 \oplus F^2 g_2$ given by
\begin{eqnarray}
 g= \pi^* g_1 + F^2 \sigma^* g_2,
\end{eqnarray}
where $F$ is a positive function on $M_1$. We denote the warped product of Riemannian manifolds $(M_1,g_1)$ and $(M_2,g_2)$ by $M_1 \times_F M_2$ and we refer to $(M_1,g_1)$ and $(M_2,g_2)$ as the base and the fiber of the product, respectively. The function $F$ is called the \textit{warping function}. If the warping function $F$ is constant then the warped product $M_1 \times_F M_2$ is a direct product, which we call as trivial warped product. The study of warped product manifolds has become a very active research subject (See \cite{diallo} and reference therein).\\
Let $\nabla$ be a linear connection in an $n$-dimensional differentiable manifold $M$. The torsion tensor $T$ and the curvature tensor $R$ of $\nabla$ are given, respectively by
\begin{eqnarray*}
 T(X,Y) &=& \nabla_X Y -\nabla_Y X -[X,Y],\\
 R(X,Y)Z &=& [\nabla_X,\nabla_Y]Z -\nabla_{[X,Y]}Z,
\end{eqnarray*}
for $X,Y,Z \in \Gamma(TM)$. The connection $\nabla$ is symmetric if its torsion tensor $T$ vanishes,
otherwise it is non-symmetric. The connection $\nabla$ is a metric connection if there is a Riemannian metric
$g$ in $M$ such that $\nabla g=0$, otherwise it is non-metric. It is known that $\nabla$ is the 
Levi-Civita connection of a Riemannian metric if the holonomy group of $\nabla$ is a subgroup of
the orthogonal group $\mathcal{O}(n)$ (\cite{schmidt}).

In 2008, Tripathi introduced a new connection in a Riemannian manifold, which is a natural extension of several 
symmetric, semi-symmetric and quarter-symmetric connections. He proved the following result.
\begin{theorem}\cite{tripathi1}
Let $(M,g)$ be an $n$-dimensional Riemannian manifold equipped with the Levi-Civita connection $\mathring{\nabla}$. 
Let $f_1,f_2$ be functions in $M$, $u,u_1,u_2$ are $1$-forms in $M$ and $\phi$ is a $(1,1)$ tensor field in $M$. Let
\begin{eqnarray}
 u(X) = g(P,X); \; u_1(X)=g(P_1,X);\;  u_2(X)=g(P_2,X) 
 \end{eqnarray}
 \begin{eqnarray}
 g(\phi(X),Y) = \Phi(X,Y)=\Phi_1(X,Y)+\Phi_2(X,Y)
\end{eqnarray}
where $P,P_1,P_2$ are vector fields on $M$ and $\Phi_1, \Phi_2$ are symmetric and skew-symmetric parts of the $(0,2)$ 
tensor $\Phi$ such that
\begin{eqnarray}\label{Equat0}
 \Phi_1(X,Y)=g(\phi_1 X,Y), \quad \Phi_2(X,Y)=g(\phi_2 X,Y).
\end{eqnarray}
Then there exists a unique connection $\nabla$ in $M$ given by
\begin{eqnarray}\label{Equat1}
 \nabla_X Y &= & \mathring{\nabla}_{X} Y +u(Y)\phi_1 X -u(X)\phi_2 Y -g(\phi_1 X,Y)P \nonumber \\
 &\quad & -f_1 \{u_1(X)Y +u_1(Y)X -g(X,Y)P_1\} \nonumber \\
 &\quad &- f_2 g(X,Y)P_2,
\end{eqnarray}
which satisfies
\begin{eqnarray}\label{Equat2}
 T(X,Y) =u(Y)\phi X -u(X)\phi Y,
\end{eqnarray}
and
\begin{eqnarray}\label{Equat3}
 (\nabla_X g)(Y,Z) &=& 2f_1u_1 (X)g(Y,Z)\nonumber \\
 &\quad & + f_2\{u_2(Y)g(X,Z)+u_2(Z)g(X,Y)\},
\end{eqnarray}
where $T$ is the torsion tensor of $\nabla$.
\end{theorem}
In this paper, we call this new connection as the \textit{Tripathi connection}. The aim of this paper is to consider warped product manifolds with a Tripathi connection and investigate relationships between the Tripathi connection of $M$ and those of the base $M_1$ and the fiber $M_2$.\\
The outline of the paper is as follows. In section 2, we recall some basic notion on warped product manifolds used in \cite{oneil}. In section 3, we establish the relationships between the Tripathi connection of a warped product and those of the base and the fiber. Finally, the last section is devoted to the special cases of Tripathi connections.

\section{Warped product Riemannian manifolds} \label{Prel}

Let $(M_1,g_1)$ and $(M_2,g_2)$ be two Riemannian manifolds and let $F>0$ be a smooth function on $M_1$. The warped
product $M=M_1 \times_F M_2$ is the product manifold $M_1 \times M_2$ equipped with the metric tensor
$$
g=\pi^* g_{1} + (F\circ \pi)^2 \sigma^* g_2.
$$ 
The manifold $M_{1}$ is called the base of $M$ and $M_{2}$ the fiber. It is easy to see that the fibers $\{p\}\times M_2=\pi^{-1}(p)$
and the leaves $M_1 \times \{q\} =\sigma^{-1}(q)$ are submanifolds of $M$, and the warped metric is characterized by
\begin{enumerate}
\item for each $q\in M_2$, the map $\pi|_{M_1 \times \{q\}}$ is a isometry onto $M_1$;
\item for each $p\in M_1$, the map $\sigma|_{\{p\} \times M_2}$ is a positive homothety onto $M_2$;
\item for each $(p,q)\in M$, the leaf $M_1 \times \{q\}$ and the fiber $\{p\}\times M_2$ are orthogonal at $(p,q)$.
\end{enumerate}
Vectors tangent to leaves are horizontal, vectors tangent to fibers are verticals. We denote by $\mathcal{H}$ the
orthogonal projection of $T_{(p,q)} M$ onto its horizontal subspace $T_{(p,q)} (M_1 \times \{q\})$, and by $\mathcal{V}$
the projection onto the vertical subspace $T_{(p,q)} (\{p\}\times M_2)$. Now we denote by $\mathrm{tan}$ for the projection $\mathcal{V}$ onto $T_{(p,q)} (\{p\}\times M_2)$ and $\mathrm{nor}$ for the projection onto $T_{(p,q)} (M_1 \times \{q\}) = (T_{(p,q)} (\{p\}\times M_2))^{\perp}$.

The relation of a warped product to the base $M_1$ is almost as simple as in the special case of a semi-Riemannian product. However, the relation to the fiber $M_2$ often involves the warping function $F$.
\begin{lemma}\cite{oneil}
If $h\in \mathcal{F}(M_1)$, then the gradient of the lift $h\circ \pi$ to $M=M_1 \times_F M_2$ is the lift to $M$
of the gradient of $h$ on $M_1$.
\end{lemma}
Thus there should be no confusion if we simplify the notation by writing $h$ for $h\circ \pi$ and $\mathrm{grad}\, h$
for $\mathrm{grad} (h\circ \pi)$.  

The Levi-Civita connection of $M= M_1 \times_F M_2$ is related to those $M_1$ and $M_2$ as follows.
\begin{proposition}\cite{oneil}
Let $M= M_1 \times_F M_2$ be a warped product manifold. Denote by $\mathring{\nabla},{}^{M_1} \mathring{\nabla}$ and 
${}^{M_2}\mathring{\nabla}$ the Levi-Civita connections on $M,M_1$ and $M_2$, respectively. Then, for any $X,Y,Z \in \Gamma(TM_1)$ 
and $U,V,W\in \Gamma(TM_2)$, we have:
\begin{enumerate}
 \item $\mathring{\nabla}_X Y$ is the lift of ${}^{M_1} \mathring{\nabla}_X Y$.
 \item $\mathring{\nabla}_X V=\nabla_V X = (X\cdot F/F)V$.
 \item The component of $\mathring{\nabla}_V W$ normal to the fibers is $-(g(V,W)/F)\mathrm{grad}F$.
 \item The component of $\mathring{\nabla}_V W$ tangent to the fibers is the lift ${}^{M_2}\mathring{\nabla}_V W$.
\end{enumerate}
\end{proposition}

\section{Warped product manifolds endowed with a Tripathi connection}

In this section, we consider warped product manifolds with respect to the Tripathi connection and we prove
the following results.

\begin{proposition}
Let $M= M_1 \times_F M_2$ be a warped product manifold. Denote by $\nabla,{}^{M_1} \nabla$ and ${}^{M_2}\nabla$ the Tripathi connections on $M,M_1$ and $M_2$ respectively. Then, for any $X,Y \in \Gamma(TM_1)$, $V,W\in \Gamma(TM_2)$ and $P,P_1,P_2 \in \Gamma(TM_1)$, we have:
\begin{enumerate}
 \item $\nabla_X Y \in \Gamma(TM_1)$ is the lift of ${}^{M_1} \nabla_X Y$ on $M_1$.
 \item $\nabla_X V=(X(F)/F)V -u(X)\phi_2 V -f_1 u_1(X)V\;\; \mbox{and} \;\;\\
 \nabla_V X = (X(F)/F)V -u(X)\phi_2 V -f_1u_1(X)V +u(X)\phi V$.
 \item $\mathrm{nor} \nabla_V W = -[g(V,W)/F]\nabla F +g(\phi_2 V,W)P +f_1g(V,W)P_1 -g(\phi V,W)P.$
 \item $\mathrm{tan} \nabla_V W $ is the lift of ${}^{M_2} \nabla_X Y$ on $M_2$.
\end{enumerate}
\end{proposition}
\begin{proof}
From the Koszul formula we have
\begin{eqnarray}\label{koszul1}
 2g(\mathring{\nabla}_X Y,Z) &=& X\cdot g(Y,Z) +Y\cdot g(X,Z) -Z\cdot g(X,Y) \nonumber \\
 &\quad & -g(X,[Y,Z]) -g(Y,[X,Z]) +g(Z,[X,Y]),
\end{eqnarray}
for all $X,Y,Z$ on $M$, where $\mathring{\nabla}$ is the Levi-Civita connection of $M$. 

(1) For any vector fields $X,Y \in \Gamma(TM_1)$ and $V\in \Gamma(TM_2)$, the equation (\ref{koszul1}) reduces to
\begin{eqnarray}\label{koszul2}
 2g(\mathring{\nabla}_X Y,V)=-V\cdot g(X,Y) + g(V,[X,Y])
\end{eqnarray}
since $g(X,V)=g(Y,V)=0$ and $[X,V]=[Y,V]=0$ from \cite{oneil}. Firstly, since $X,Y$ are lifts
from $M_1$ and $V$ is vertical, $g(X,Y)$ is constant on fibers, which means that $V\cdot g(X,Y)=0$. On the other hand,
$[X,Y]$ is tangent to leaves, hence $g(V,[X,Y])=0$. From equation (\ref{Equat1}), the equation (\ref{koszul2}) becomes
\begin{eqnarray}\label{koszul3}
g(\nabla_X Y,V) &=& u(Y)g(\phi_1 X,V) -u(X)g(\phi_2 Y,V) +g(\phi_1 X,Y)g(P,V)\nonumber \\
 &-& f_1 \Big[u_1 (X)g(Y,V) +u_1 (Y)g(X,V) -g(X,Y)g(P_1,V)\Big]\nonumber \\
 &- & f_2 g(X,Y)g(P_2,V)
\end{eqnarray}
for any vector fields $X,Y \in \Gamma(TM_1)$ and $V\in \Gamma(TM_2)$. From the equation (\ref{Equat0}), since 
$X,Y$ are lifts from $M_1$ and $V$ is vertical, we have
$$
g(\phi_1 X,V) =\Phi_1(X,V) =0 \quad \mbox{and}\quad  g(\phi_2 Y,V)=\Phi_2(Y,V) =0.
$$
Hence, the equation (\ref{koszul3}) reduces to 
\begin{eqnarray}\label{koszul4}
g(\nabla_X Y,V) &=& -g(\phi_1 X,Y)g(P,V) +f_1 g(X,Y)g(P_1,V)\nonumber \\
&\quad & -f_2 g(X,Y)g(P_2,V).
\end{eqnarray}
Since $P,P_1,P_2 \in \Gamma(TM_1)$, from (\ref{koszul4}) we get
$$
g(\nabla_X Y,V) =0.
$$
Thus $g(\nabla_X Y,V)=0$ for all $V\in \Gamma(TM_2)$, so $\nabla_X Y$ is horizontal. This gives us (1) of the Proposition.

(2) By definition of the covariant derivative with respect to the Tripathi metric connection, we have
$$
g(\nabla_X V,Y) = X\cdot g(V,Y) - g(V,\nabla_X Y)
$$
for all vector fields $X,Y$ on $M_1$ and $V$ on $M_2$. Since, $g(V,Y)=0$ and from (\ref{koszul4}), the above equation 
turns into
\begin{eqnarray}\label{koszul5}
 g(\nabla_X V,Y) &=& g(\phi_1 X,Y)g(P,V) -f_1 g(X,Y)g(P_1,V)\nonumber \\
&\quad & +f_2 g(X,Y)g(P_2,V).
\end{eqnarray}
Since $P,P_1,P_2 \in \Gamma(TM_1)$, we get $g(\nabla_X V,Y)=0$ for all $Y\in \Gamma(TM_1)$, so $\nabla_X V$ is vertical.
From the equation (\ref{koszul1}) and the definition of the Tripathi connection (\ref{Equat1}), we obtain
\begin{align*}
&2g(\nabla_X V,W) = X\cdot g(V,W) +V\cdot g(X,W) -W\cdot g(X,V) \nonumber \\
 & -g(X,[V,W]) -g(V,[X,W]) +g(W,[X,V]) \nonumber \\
 & +2u(V)g(\phi_1 X,W) -2u(X)g(\phi_2 V,W) -2g(\phi_1 X,V)g(P,W)\nonumber \\
 & -2f_1 \Big[u_1 (X)g(V,W) +u_1 (V)g(X,W) -g(X,V)g(P_1,W)\Big]\nonumber \\
 & -2f_2 g(X,V)g(P_2,W),
\end{align*}
for any vector fields $X\in \Gamma(TM_1)$ and $V,W\in \Gamma(TM_2)$. Since $g(X,W)=g(X,V)=0$ and $[X,V]=[X,W]=0$,
we have
\begin{align}\label{koszul6}
&2g(\nabla_X V,W) = X\cdot g(V,W)  -g(X,[V,W])  \nonumber \\
 & +2u(V)g(\phi_1 X,W) -2u(X)g(\phi_2 V,W) -2g(\phi_1 X,V)g(P,W)\nonumber \\
 &  -2f_1 \Big[u_1 (X)g(V,W) \Big].
\end{align}
Since $X, \phi_1 X$ are horizontals and $\phi_2 V, [V,W]$ are verticals, we have:
$$
g(X,[V,W])=0, \quad g(\phi_1 X,W)=g(\phi_1 X,V)=0.
$$
Hence we find
\begin{eqnarray}\label{koszul7}
 2g(\nabla_X V,W) &=& X\cdot g(V,W) -2u(X)g(\phi_2 V,W) \nonumber \\
 &\quad & -2f_1 u_1 (X)g(V,W).
\end{eqnarray}
From the definition of the warped product metric, we have
$$
g(V,W)(p,q)=(F\circ \pi)^2 (p,q)g_{M_2}(V_q,W_q).
$$
Then by making use of $F$ instead of $F\circ \pi$, we get
$$
g(V,W)=F^2 (g_{M_2} (V,W)\circ \sigma).
$$
Hence, we get
$$
X\cdot g(V,W)=2 F(X\cdot F)(g_{M_2}(V,W)\circ \sigma) +F^2 X\cdot (g_{M_2} (V,W)\circ \sigma).
$$
Since the term $g_{M_2} (V,W)\circ \sigma$ is constant on leaves, the above equation turns into
\begin{eqnarray}\label{equation8}
 X\cdot g(V,W)=2 F (X\cdot F)(g_{M_2}(V,W)\circ \sigma).
\end{eqnarray}
By making use (\ref{equation8}) in (\ref{koszul7}), we obtain
\begin{eqnarray*}
 g(\nabla_X V,W) = (X\cdot F/F)(g(V,W) -u(X)g(\phi_2 V,W) -f_1 u_1 (X)g(V,W).
\end{eqnarray*}
Taking $P,P_1,P_2 \in \Gamma(TM_1)$, we have
\begin{eqnarray}\label{Equat4}
 \nabla_X V = (X\cdot F/F)V -u(X)\phi_2 V -f_1 u_1 (X)V.
\end{eqnarray}
By using the definition of the torsion tensor of the Tripathi connection, we have
\begin{eqnarray*}
 g(\nabla_X V,W) =g(\nabla_V X,W) -u(X)g(\phi V,W)
\end{eqnarray*}
which means that
\begin{align}
 g(\nabla_V X,W) &=g([(X\cdot F/F)V -u(X)\phi_2 V -f_1u_1(X)V],W)\nonumber\\
 &+u(V)g(\phi V,W).
\end{align} 
Then we get
\begin{eqnarray}
 \nabla_V X = (X\cdot F/F)V -u(X)\phi_2 V -f_1u_1(X)V +u(X)\phi V,
\end{eqnarray}
so we obtain the (2) of the Proposition.

(3) From the covariant derivative with respect to the Tripathi connection, we have
$$
V\cdot g(X,V) =g(\nabla_V X,W) + g(X,\nabla_V W),
$$
for any vector $X$ on $M_1$ and $V,W$ on $M_2$. Since $g(X,V)=0$, the above equation reduces to
\begin{align*}
 &g(\nabla_V W,X)  =  -g(\nabla_V X,W)\\
 &=  -g([(X\cdot F/F)V -u(X)\phi_2 V -f_1u_1(X)V +u(X)\phi V],W)\\
 &=  -g([(g(\nabla F,X)/F)V - g(P,X)\phi_2 V \\
 &  -f_1 g(P_1,X)V +g(P,X)\phi V],W).
\end{align*}
Then, we get
\begin{eqnarray*}
 \mathrm{nor}\nabla_V W &=& -[g(V,W)/F]\nabla F +g(\phi_2 V,W)P \\
 &\quad & +f_1g(V,W)P_1 -g(\phi V,W)P.
\end{eqnarray*}

(4) Since $V$ and $W$ are tangent to all fibers, then on a fiber, $\mathrm{tan} \nabla_V W$ is the
fiber covariant derivative applied to the restrictions of $V$ and $W$ to that fiber. See \cite{oneil}
for more detail. Thus the proof of the proposition is completed.
\end{proof}
\begin{proposition}
Let $M= M_1 \times_F M_2$ be a warped product manifold. Denote by $\nabla,{}^{M_1} \nabla$ and ${}^{M_2}\nabla$ the 
Tripathi connections on $M,M_1$ and $M_2$ respectively. Then, for any $X,Y \in \Gamma(TM_1)$, $V,W\in \Gamma(TM_2)$ and
$P,P_1,P_2 \in \Gamma(TM_2)$, we have:
\begin{enumerate}
 \item $\mathrm{nor} \nabla_X Y$ is the lift of $\nabla_X Y$ on $M_1$.
 \item $\mathrm{tan} \nabla_X Y =-g(\phi_1 X,Y)P + f_1 g(X,Y)P_1 -f_2 g(X,Y)P_2$.
 \item $\mathrm{tan} \nabla_X V =(X\cdot F/F)V -u(X)\phi_2 V -f_1 u(X)V $and
 $\mathrm{nor} \nabla_X V= g(P,V)\phi_1 X -f_1g(P_1,V)X +f_2g(P_2,V)X.$
 \item \begin{eqnarray*} 
 \nabla_V X &=& (X\cdot F/F)V -u(X)\phi_2 V -f_1 u(X)V \nonumber \\
 & \quad & +g(P,V)\phi_1 X -f_1g(P_1,V)X +f_2g(P_2,V)X \nonumber \\
 &\quad & -u(V)\phi X +u(X)\phi V.
 \end{eqnarray*}
 \item $\mathrm{nor} \nabla_V W = -[g(V,W)/F] \nabla F + g(\phi_2 V,W)P +f_1 g(V,W)P -g(\phi V,W)P$.
 \item $\mathrm{tan} \nabla_V W$ is the lift of $\nabla_V W$ on $M_2$.
\end{enumerate}
\end{proposition}

\begin{proof}
From equation (\ref{koszul4}), we obtain 
\begin{eqnarray}\label{koszul8}
 g(\nabla_X Y,V) &=& -g(\phi_1 X,Y)g(P,V) +f_1 g(X,Y)g(P_1,V)\nonumber \\
&\quad & -f_2 g(X,Y)g(P_2,V).
\end{eqnarray}
Since $P,P_1,P_2 \in \Gamma(TM_2)$, from (\ref{koszul8}) we get
$$
\nabla_X Y =-g(\phi_1 X,Y)P + f_1 g(X,Y)P_1 -f_2 g(X,Y)P_2
$$
which gives us (1) and (2) of the Proposition.\\
Similarly, from equation (\ref{koszul5}), we obtain
\begin{eqnarray}\label{koszul9}
 g(\nabla_X V,Y) &=& g(\phi_1 X,Y)g(P,V) -f_1 g(X,Y)g(P_1,V)\nonumber \\
&\quad & +f_2 g(X,Y)g(P_2,V).
\end{eqnarray}
Since $P,P_1,P_2 \in \Gamma(TM_2)$ and by the use of (\ref{Equat4}), we get
\begin{eqnarray}\label{3.13}
 \nabla_X V &=& (X\cdot F/F)V -u(X)\phi_2 V -f_1 u(X)V \nonumber \\
 & \quad & +g(P,V)\phi_1 X -f_1g(P_1,V)X +f_2g(P_2,V)X,
\end{eqnarray}
which implies that
\begin{eqnarray*}
 \mathrm{tan} \nabla_X V =(X\cdot F/F)V -u(X)\phi_2 V -f_1 u(X)V 
 \end{eqnarray*} 
and
 \begin{eqnarray*}
 \mathrm{nor} \nabla_X V= g(P,V)\phi_1 X -f_1g(P_1,V)X +f_2g(P_2,V)X.
\end{eqnarray*}
Hence we have (3).\\
From the definition of the torsion tensor, we have
$$
\nabla_V X =\nabla_X V -T(X,V).
$$
By by making use (\ref{Equat2}) and (\ref{3.13}), we get
\begin{eqnarray*}
 \nabla_V X &=& (X\cdot F/F)V -u(X)\phi_2 V -f_1 u(X)V \nonumber \\
 & \quad & +g(P,V)\phi_1 X -f_1g(P_1,V)X +f_2g(P_2,V)X \nonumber \\
 &\quad & -u(V)\phi X +u(X)\phi V
\end{eqnarray*}
which complete the proof of $(4)$. Similarly taking $P,P_1,P_2 \in \Gamma(TM_2)$, we obtain
\begin{eqnarray*}
 g(\nabla_V W,X) =  -g(\nabla_V X,W).
\end{eqnarray*}
We obtain that
$$
\mathrm{nor}\nabla_V W = -[g(V,W)/F] \nabla F + g(\phi_2 V,W)P +f_1 g(V,W)P -g(\phi V,W)P
$$
Hence, we complete the proof of the Proposition.
\end{proof}

\section{Particular cases}

 
A linear connection is said to be a semi-symmetric connection if its torsion 
tensor $T$ is of the form
$$
T(X,Y)=u(Y)X -u(X)Y,
$$
where $u$ is a $1$-form. The notion of a semi-symmetric metric connection on a Riemannian manifold was introduced by H. A. Hayden in \cite{hayden} and K. 
Yano in \cite{yano} studies some of its properties. In equation (\ref{Equat1}), when $f_1=0=f_2, \phi=\mathrm{Id}$, then we obtain a semi-symmetric metric connection given by (K. Yano \cite{yano})
\begin{eqnarray*}
 \nabla_X Y =\mathring{\nabla}_X Y +u(Y)X -g(X,Y)P.
\end{eqnarray*}
We have the following results:
\begin{corollary}\cite{sular}
Let $M= M_1 \times_F M_2$ be a warped product manifold. Denote by $\nabla,{}^{M_1} \nabla$ and ${}^{M_2}\nabla$ the semi-symmetric metric connections on $M,M_1$ and $M_2$ respectively. Then, for any $X,Y \in \Gamma(TM_1)$, $V,W\in \Gamma(TM_2)$ and $P \in \Gamma(TM_1)$, we have:
\begin{enumerate}
 \item $\nabla_X Y \in \Gamma(TM_1)$ is the lift of ${}^{M_1} \nabla_X Y$ on $M_1$.
 \item $\nabla_X V=(X\cdot F/F)V \quad \mbox{and} \quad \nabla_V X = (X\cdot F/F)V +u(X)\phi V$.
 \item $\mathrm{nor} \nabla_V W = -[g(V,W)/F]\nabla F  -g(V,W)P.$
 \item $\mathrm{tan} \nabla_V W $ is the lift of ${}^{M_2} \nabla_X Y$ on $M_2$.
\end{enumerate}
\end{corollary} 
\begin{corollary}\cite{sular}
Let $M= M_1 \times_F M_2$ be a warped product manifold. Denote by $\nabla,{}^{M_1} \nabla$ and ${}^{M_2}\nabla$ the 
semi-symmetric metric connections on $M,M_1$ and $M_2$ respectively. Then, for any $X,Y \in \Gamma(TM_1)$, $V,W\in \Gamma(TM_2)$ 
and $P \in \Gamma(TM_2)$, we have:
\begin{enumerate}
 \item $\mathrm{nor} \nabla_X Y$ is the lift of $\nabla_X Y$ on $M_1$.
 \item $\mathrm{tan} \nabla_X Y =-g(X,Y)P$.
 \item $\mathrm{tan} \nabla_X V =(X\cdot F/F)V \quad \mbox{and}\quad \mathrm{nor} \nabla_X V= g(U,V)X$.
 \item $\nabla_V X = (X\cdot F/F)V$.
 \item $\mathrm{nor} \nabla_V W = -[g(V,W)/F] \nabla F$.
 \item $\mathrm{tan} \nabla_V W$ is the lift of $\nabla_V W$ on $M_2$.
\end{enumerate}
\end{corollary}

In \cite{agashe}, Agashe and Chafle introduced the notion of semi-symmetric non-metric connection and studies some of its properties.
In equation (\ref{Equat1}), if $f_1=0, f_2=-1, \phi=\mathrm{Id}, u_2=u$, then we obtain a semi-symmetric
non-metric given by (N. S. Agashe and M. R. Chafle \cite{agashe})
\begin{eqnarray*}
\nabla_X Y =\mathring{\nabla}_X Y +u(Y)X. 
\end{eqnarray*}
This connection satisfies
$$
(\nabla_X g)(Y,Z)=-u(Y)g(X,Z) -u(Z)g(X,Y).
$$
We have the following results:
\begin{corollary}\cite{ozgur}
Let $M= M_1 \times_F M_2$ be a warped product manifold. Denote by $\nabla,{}^{M_1} \nabla$ and ${}^{M_2}\nabla$ the 
semi-symmetric non-metric connections on $M,M_1$ and $M_2$ respectively. Then, for any $X,Y \in \Gamma(TM_1)$, $V,W\in \Gamma(TM_2)$ 
and $P \in \Gamma(TM_1)$, we have:
\begin{enumerate}
 \item $\nabla_X Y \in \Gamma(TM_1)$ is the lift of ${}^{M_1} \nabla_X Y$ on $M_1$.
 \item $\nabla_X V=(X\cdot F/F)V \quad \mbox{and} \quad \nabla_V X = (X\cdot F/F)V +u(X)\phi V$.
 \item $\mathrm{nor} \nabla_V W = -[g(V,W)/F]\nabla F$.
 \item $\mathrm{tan} \nabla_V W $ is the lift of ${}^{M_2} \nabla_X Y$ on $M_2$.
\end{enumerate}
\end{corollary} 
 \begin{corollary}\cite{ozgur}
Let $M= M_1 \times_F M_2$ be a warped product manifold. Denote by $\nabla,{}^{M_1} \nabla$ and ${}^{M_2}\nabla$ the 
semi-symmetric non-metric connections on $M,M_1$ and $M_2$ respectively. Then, for any $X,Y \in \Gamma(TM_1)$, $V,W\in \Gamma(TM_2)$ 
and $P \in \Gamma(TM_2)$, we have:
\begin{enumerate}
 \item $\nabla_X Y$ is the lift of $\nabla_X Y$ on $M_1$.
 \item $\mathrm{tan} \nabla_X V =(X\cdot F/F)V \quad \mbox{and}\quad \mathrm{nor} \nabla_X V= g(P,V)X$.
 \item $\nabla_V X = (X\cdot F/F)V$.
 \item $\mathrm{nor} \nabla_V W = -[g(V,W)/F] \nabla F$.
 \item $\mathrm{tan} \nabla_V W$ is the lift of $\nabla_V W$ on $M_2$.
\end{enumerate}
\end{corollary} 
In \cite{ozgur} (respectively, in  \cite{sular}), formulas relating the Riemannian curvature tensors, the Ricci curvatures and the scalar curvatures of the Levi-Civita connection and the
semi-symmetric non-metric connection (respectively, metric connection) of a warped product Riemannian manifold are presented. The authors characterize also warped product manifolds $M\times_f I$ and $I\times_f M$ with a one dimensional manifold $I$ which are Einstein manifolds with respect to a semi-symmetric non-metric connection (respectively, metric connection).

In \cite{golab}, S. Golab  defined and studied quarter-symmetric linear connections in a differentiable manifold. A linear connection $\nabla$ on an $n$-dimensional Riemannian manifold $(M,g)$ is called a \textit{quarter-symmetric connection} if its torsion tensor $T$ of the connection $\nabla$ satisfies
$$
T(X,Y)=u(Y)\phi X - u(X)\phi Y
$$
where $u$ is a $1$-form and $\phi$ is a $(1,1)$ tensor field. In Particular, if $\phi X=X$, then the quarter-symmetric connection reduces to the semi-symmetric connection. Thus the notion of quarter-symmetric connection generalizes the idea of the semi-symmetric connection. If moreover, a quarter-symmetric connection $\nabla$ satisfies the conditions
$$
(\nabla_X g)(Y,Z)=0
$$
for $X,Y,Z\in \Gamma(TM)$, then $\nabla$ is said to be a quarter-symmetric metric connection, otherwise it is said to be a quarter-symmetric non-metric connection. Various properties of quarter-symmetric metric connections have studied in \cite{golab,mishra,rastogi,sular2,yano2}.
In equation (\ref{Equat1}), if $f_1=0=f_2$. Then we obtain a quarter-symmetric metric connection 
$\nabla$ given by (K. Yano and T. Imai \cite{yano2})
$$
\nabla_X Y =\mathring{\nabla}_X Y +u(Y)\phi_1 X -u(X)\phi_2 Y -g(\phi_1 X,Y)P.
$$
We have the following:

\begin{corollary}
Let $M= M_1 \times_F M_2$ be a warped product manifold. Denote by $\nabla,{}^{M_1} \nabla$ and ${}^{M_2}\nabla$ the 
quarter-symmetric connections on $M,M_1$ and $M_2$ respectively. Then, for any $X,Y \in \Gamma(TM_1)$, $V,W\in \Gamma(TM_2)$ and
$P\in \Gamma(TM_1)$, we have:
\begin{enumerate}
 \item $\nabla_X Y \in \Gamma(TM_1)$ is the lift of ${}^{M_1} \nabla_X Y$ on $M_1$.
 \item $\nabla_X V=(X\cdot F/F)V -u(X)\phi_2 V$ and $\nabla_V X = (X\cdot F/F)V -u(X)\phi_2 V +u(X)\phi V$.
 \item $\mathrm{nor} \nabla_V W = -[g(V,W)/F]\nabla F  -g(\phi V,W)P.$
 \item $\mathrm{tan} \nabla_V W $ is the lift of ${}^{M_2} \nabla_X Y$ on $M_2$.
\end{enumerate}
\end{corollary}
\begin{corollary}
Let $M= M_1 \times_F M_2$ be a warped product manifold. Denote by $\nabla,{}^{M_1} \nabla$ and ${}^{M_2}\nabla$ the 
quarter-symmetric connections on $M,M_1$ and $M_2$ respectively. Then, for any $X,Y \in \Gamma(TM_1)$, $V,W\in \Gamma(TM_2)$ and
$P\in \Gamma(TM_2)$, we have:
\begin{enumerate}
 \item $\mathrm{nor} \nabla_X Y$ is the lift of $\nabla_X Y$ on $M_1$.
 \item $\mathrm{tan} \nabla_X Y =-g(\phi_1 X,Y)P$.
 \item $\mathrm{tan} \nabla_X V =(X\cdot F/F)V -u(X)\phi_2 V \quad \mbox{and}\quad
 \mathrm{nor} \nabla_X V= g(U,V)\phi_1 X$.
 \item \begin{align*} 
 \nabla_V X &=  (X\cdot F/F)V -u(X)\phi_2 V   +g(U,V)\phi_1 X  \nonumber \\
 & -u(V)\phi X +u(X)\phi V.
 \end{align*}
 \item $\mathrm{nor} \nabla_V W = -[g(V,W)/F] \nabla F + g(\phi_2 V,W)U -g(\phi V,W)U$.
 \item $\mathrm{tan} \nabla_V W$ is the lift of $\nabla_V W$ on $M_2$.
\end{enumerate}
\end{corollary}

The properties of quarter-symmetric non-metric connections have been also studied by \cite{dubey,sengupta} and by many others authors. In equation (\ref{Equat1}), if $f_1=0, f_2\neq 0, \phi_1=0$. Then we obtain a quarter-symmetric non-metric 
connection $\nabla$ given by
$$
\nabla_X Y=\mathring{\nabla}_X Y -u(X)\phi X -f_2 g(X,Y)P.
$$
This connection satisfies
$$
\Big(\nabla_X g\Big)(Y,Z)=f_2 \{u_2(Y)g(X,Z) +u_2(Z)g(X,Y)\}.
$$
We have the following:
\begin{corollary}
Let $M= M_1 \times_F M_2$ be a warped product manifold. Denote by $\nabla,{}^{M_1} \nabla$ and ${}^{M_2}\nabla$ the 
Tripathi connections on $M,M_1$ and $M_2$ respectively. Then, for any $X,Y \in \Gamma(TM_1)$, $V,W\in \Gamma(TM_2)$ and
$P,P_1,P_2 \in \Gamma(TM_1)$, we have:
\begin{enumerate}
 \item $\nabla_X Y \in \Gamma(TM_1)$ is the lift of ${}^{M_1} \nabla_X Y$ on $M_1$.
 \item $\nabla_X V=(X\cdot F/F)V -u(X)\phi V$ and $\nabla_V X = (X\cdot F/F)V$.
 \item $\mathrm{nor} \nabla_V W = -[g(V,W)/F]\nabla F +g(\phi_2 V,W)P_1  -g(\phi V,W)P.$
 \item $\mathrm{tan} \nabla_V W $ is the lift of ${}^{M_2} \nabla_X Y$ on $M_2$.
\end{enumerate}
\end{corollary}
\begin{corollary}
Let $M= M_1 \times_F M_2$ be a warped product manifold. Denote by $\nabla,{}^{M_1} \nabla$ and ${}^{M_2}\nabla$ the 
Tripathi connections on $M,M_1$ and $M_2$ respectively. Then, for any $X,Y \in \Gamma(TM_1)$, $V,W\in \Gamma(TM_2)$ and
$P,P_1,P_2 \in \Gamma(TM_2)$, we have: 
\begin{enumerate}
 \item $\mathrm{nor} \nabla_X Y$ is the lift of $\nabla_X Y$ on $M_1$.
 \item $\mathrm{tan} \nabla_X Y =-f_2 g(X,Y)P_2$.
 \item $\mathrm{tan} \nabla_X V =(X\cdot F/F)V -u(X)\phi V $ and $\mathrm{nor} \nabla_X V= f_2g(U_2,V)X.$
 \item \begin{align*} 
 \nabla_V X &=  (X\cdot F/F)V -u(X)\phi V   +f_2g(U_2,V)X \nonumber \\
 & -u(V)\phi X +u(X)\phi V.
 \end{align*}
 \item $\mathrm{nor} \nabla_V W = -[g(V,W)/F] \nabla F + g(\phi_2 V,W)U -g(\phi V,W)U$.
 \item $\mathrm{tan} \nabla_V W$ is the lift of $\nabla_V W$ on $M_2$.
\end{enumerate}
\end{corollary} 
As a perspective, the authors will try to establish formulas relating the Riemannian curvature tensors, the Ricci curvatures and the scalar curvatures of the Levi-Civita connection and the semi-symmetric non-metric connection of a warped product Riemannian manifold.

\end{document}